\documentclass[12pt]{amsart}
\usepackage{graphicx} 

\usepackage{amsmath}
\usepackage{amssymb}
\usepackage{amsthm}
\usepackage{xcolor}
\usepackage{hyperref}
\numberwithin{equation}{section}
\usepackage[all]{xy}

\newcommand{\ra}{\rightarrow}

\newcommand{\bF}{\mathbb{F}}
\newcommand{\bG}{\mathbb{G}}

\newcommand{\bP}{\mathbb{P}}
\newcommand{\bQ}{\mathbb{Q}}

\newcommand{\bZ}{\mathbb{Z}}

\newcommand{\cC}{\mathcal{C}}
\newcommand{\cF}{\mathcal{F}}

\newcommand{\cO}{\mathcal{O}}
\newcommand{\cP}{\mathcal{P}}
\newcommand{\cQ}{\mathcal{Q}}

\newcommand{\fS}{\mathfrak{S}}

\newcommand{\oM}{\overline{M}}

\newcommand{\Aut}{\mathrm{Aut}}
\newcommand{\Bl}{\operatorname{Bl}}

\newcommand{\GL}{\mathrm{GL}}

\newcommand{\pardeg}{\mathrm{pardeg}}
\newcommand{\PGL}{\mathrm{PGL}}
\newcommand{\Pic}{\mathrm{Pic}}

\newcommand{\SL}{\mathrm{SL}}

\newcommand{\Spin}{\operatorname{Spin}}

\theoremstyle{plain}
\newtheorem{prop}{Proposition}

\newtheorem{theo}[prop]{Theorem}
\newtheorem{coro}[prop]{Corollary}

\theoremstyle{definition}
\newtheorem{defi}[prop]{Definition}

\newtheorem{ques}[prop]{Question}

\theoremstyle{remark}
\newtheorem{rema}[prop]{Remark}
\newtheorem{exam}[prop]{Example}

\title{Moduli of conic surfaces over the line}
\author{Brendan Hassett}
\email{brendan\underline{ }hassett@brown.edu}
\author{Amanda Hernandez}
\email{amanda\underline{ }hernandez@brown.edu}
\address{Department of Mathematics, Brown University, Providence, Rhode Island, USA}
\date{July 13, 2026}

\begin{document}

\begin{abstract}
We study moduli spaces and monodromy groups for surfaces 
fibered in conics over the projective line, with a view toward
explicit presentations and arithmetic applications.
\end{abstract}

\maketitle

The last decade has seen major advances in our understanding
of moduli spaces of Fano varieties through new techniques 
like K-stability. del Pezzo surfaces are an important guiding
example -- see \cite{OSS}, for instance -- where modern
moduli spaces may be related to approaches via
Geometric Invariant Theory (GIT). Surfaces with
conic bundle structures, beyond the del Pezzo case, 
have received less attention.  Some singular del Pezzo
surfaces also come with natural conic fibrations
e.g. degree-$2a$ hypersurfaces in $\bP(1,1,a,a)$ \cite{LP}.
Threefolds with conic fibrations over the plane have been
studied in specific cases \cite{DJKQ}. 

Here we will focus on surfaces fibered in conics over $\bP^1$.
These have long been studied with a view
toward Diophantine problems e.g.~\cite[\S 2.6]{CTSansuc}. 
Our target application is 
over finite fields, so we avoid using characteristic-zero 
or analytic techniques. We focus on explicit constructions
of parameter spaces for these varieties, focusing on 
presentations where sampling can be done quickly and efficiently.
A forthcoming paper by the second author \cite{Hernandez} 
will analyze the statistical behavior of rationality 
for these varieties over large finite fields. 

Section~\ref{sect:BCB} presents conic bundles and 
their classification, focusing both on equations and
inductive structures under birational maps. We turn
to the Picard group in Section~\ref{sect:LWG}, presenting
Weyl-group symmetries arising from the intersection form.
Moduli spaces, and associated birational transformations,
are the subject of Section~\ref{sect:MCTM}:
Theorem~\ref{theo:modulistack} is our main result; Corollary~\ref{coro:monodromy} gives its applications to monodromy.
Sections \ref{sect:RWPB} and \ref{sect:MCI} link our
perspective with other moduli frameworks: Parabolic bundles
and complete intersections.

\subsection*{Acknowledgments:} We are grateful to 
Ana-Maria Castravet and Kaiqi Yang for conversations
about this project.
The first author received partial support from
U.S. National Science Foundation award 2424556
and Simons Foundation Award 546235. The second author received partial support from U.S. Department of Education Graduate Assistance in Areas of National Need (GAANN) award P200A240015-25.

\section{Background on conic bundles}
\label{sect:BCB}

Let $k$ be a field of characteristic different from two.

\subsection{Basic notions}

\begin{defi} \label{defi:goodconic}
A {\em good conic bundle over $\bP^1$} consists of a 
smooth projective surface $X$ and a dominant morphism
$\phi:X \ra \bP^1$ such that:
\begin{enumerate}
\item{the geometric generic fiber of $\phi$ is $\bP^1$:}
\item{non-smooth fibers are isomorphic to two 
copies of $\bP^1$ meeting in a node.}
\end{enumerate}
\end{defi}
Being ``good'' is an open condition for families of such conic bundles.
\begin{prop}
The enumerated conditions of Definition~\ref{defi:goodconic}
are equivalent to
\begin{enumerate}
\item[1'.]{the geometric generic fiber of $\phi$ is
smooth and connected;}
\item[2'.]{$\omega_{\phi}^{-1}$ is ample relative to $\phi$.}
\end{enumerate}
Furthermore, $\omega_{\phi}^{-1}$ is very ample and
globally generated relative to $\phi$, inducing an embedding
$$
\xymatrix{ X \ar@{^{(}->}[rr] \ar[rd]_{\phi} &  &\bP((\phi_*\omega_{\phi}^{-1})^{\vee}) \ar^
{\varpi}[ld] \\
 &\bP^1 &
}
$$
where $\pi$ is a $\bP^2$-bundle over $\bP^1$.
The discriminant divisor of $\phi$ is reduced.  
\end{prop}
\begin{proof}
Keep in mind that $\phi$ is flat and projective.

Assume (1) and (2). Condition (1') is a weakening of
(1). Our assumptions imply that, on each fiber $X_p=\phi^{-1}(p)$, global sections of $\omega^{-1}_{X_p}$ imbed 
$X_p$ as a conic in $\bP^2$; further, 
$\omega^{-1}_{X_p}$ has no higher 
cohomology. By cohomology-and-base-change, we see that
$\omega_{\phi}^{-1}$ is very ample and globally generated
relative to $\phi$, inducing an embedding as a conic divisor in a
$\bP^2$-bundle. Condition (2') follows.

Assume (1') and (2'). A smooth projective
geometrically-connected curve with ample anti-canonical
classes is necessarily a smooth conic. 
A Gorenstein projective curve $C$ with $\Gamma(\cO_C)=k$
is a plane curve of degree two \cite[Lem.~53.10.3, Tag~0C6N]{stacks-project}. Of course, a Cartier divisor in a smooth
surface is Gorenstein. The case of a double line is
precluded, as such a fiber could only occur for singular $X$;
indeed, the smoothness of $X$ guarantees that the discriminant
of $\phi$ is reduced.  
The cases enumerated in (1) and (2) are the remaining
possibilities.
\end{proof}

\begin{defi}
Consider a good conic bundle $\phi:X \ra \bP^1$ with
$$(\phi_*\omega_{\phi}^{-1})^{\vee} 
\simeq \cO_{\bP^1}(a_0)  \oplus
\cO_{\bP^1}(a_1) \oplus \cO_{\bP^1}(a_2), \quad a_0 \le a_1 \le a_2.$$
The triple $(a_0,a_1,a_2) \in \bZ^3$ is the {\em
splitting type} of the bundle; it is
{\em balanced} if $a_2-a_0 \le 1$.  
\end{defi}
Being ``balanced'' is an open condition on
conic bundles, as generic vector
bundles on $\bP^1$ have this property. 

\begin{prop} \label{prop:disc}
Let $\phi:X \ra \bP^1$ be a good conic bundle of splitting type
$(a_0,a_1,a_2)$. Then the discriminant of 
$\phi$ has degree $a_0+a_1+a_2$
and 
$$\chi(X,T_X) = 6 - 2(a_0+a_1+a_2).$$
\end{prop}
\begin{proof}
Recall that the discriminant divisor is reduced; its degree
$n$ is the number of degenerate fibers of $\phi$. 
The topological Euler characteristic $\chi(X)=4+n$.
By the Noether formula and the Gauss-Bonnet theorem, 
$n=8-c_1(\omega_X)^2$. 
The Grothendieck-Riemann-Roch formula 
\cite[Th.~15.2]{FultonIT} implies
$$ \deg(\phi_*\omega_{\phi}^{-1}) = c_1(\omega_X)^2 - 8.$$
Combining these gives the first statement.
Riemann-Roch on $X$ gives the formula for its tangent bundle.
\end{proof}
\begin{coro} \label{coro:balanceforn}
For a balanced good conic bundle $\phi:X \ra \bP^1$
with $n$ degenerate fibers we have
$$(\phi_*\omega_{\phi}^{-1})^{\vee}=
\cO_{\bP^1}(\lfloor n/3 \rfloor) \oplus
\cO_{\bP^1}(\lfloor (n+1)/3 \rfloor) \oplus
\cO_{\bP^1}(\lfloor (n+2)/3 \rfloor).$$
\end{coro}

We analyze conic bundles under birational
morphisms:
\begin{prop}\label{prop:birat}
Given a birational morphism of good conic
bundles
$$
\xymatrix{ X \ar[rr]^{\beta} \ar[rd]_{\phi}
& & X' \ar[ld]^{\varphi} \\
& \bP^1 & 
}
$$
we have an inclusion
$$
\phi_*\omega_{\phi}^{-1} 
\hookrightarrow \varphi_*\omega_{\varphi}^{-1}$$
and $-a_2(\phi) \le -a_2(\varphi)$.
If $X$ has a section of self-intersection $-r$
with $r\ge 0$ then $a_2(\phi)\ge r$.  
\end{prop}
\begin{proof}
The inclusion 
$\omega_{\phi}^{-1} \hookrightarrow
\beta^*\omega_{\varphi}^{-1}
$
is obtained by iteratively applying the blowup formula for surfaces: If $b:\Bl_s(S) \ra S$ is 
the blowup of a smooth surface with exceptional curve $E$ then
$$\omega_{\Bl_s(S)}=b^*\omega_S(E).$$
Taking direct images gives the inclusion of sheaves
on $\bP^1$. 
It follows that the most negative summand of 
$\phi_*\omega_{\phi}^{-1}$ is no greater than the 
most negative summand of the 
$\varphi_*\omega_{\varphi}^{-1}$; this
gives the first inequality.

Suppose $X$ admits a section $\Sigma$ of self-intersection $-r$. Blowing down fibral components disjoint from $\Sigma$
induces birational morphsm $\beta:X \ra \bF_r$ to a Hirzebruch
surface. Since $a_2(\bF_r/\bP^1)=r$ we get
the final assertion.
\end{proof}

\subsection{Key constructions}
\begin{exam} \label{exam:types}
In each case below $\phi$ is projection onto the first factor.

\begin{enumerate}
\item \cite[\S 3]{Iskpencil}
Let $X=\{G=0\} \subset \bP^1\times \bP^2$ be smooth of
bidegree $(a,2)$. Then the splitting type is $(a,a,a)$.  
\item
Let $X \subset \bP^1\times \bP^3$ be a smooth complete 
intersection of forms $F$ and $G$ of bidegrees $(1,1)$ and $(a,2)$. 
Assume that $Y:=\{F=0\}$ is irreducible, which implies
$$Y\simeq \bP(\cO_{\bP^1}^2 \oplus \cO_{\bP^1}(-1)).$$
The sheaf $(\pi_1)_*(\cO_Y(0,1))$ is presented
$$ 0 \ra \cO_{\bP^1}(-1) \ra \cO_{\bP^1}^4 \ra 
(\pi_1)_*(\cO_Y(0,1)) \ra 0$$
whence 
$$(\pi_1)_*\cO_Y(0,1) \simeq \cO_{\bP^1}^2 \oplus \cO_{\bP^1}(1).$$
It follows that 
$$\phi_* \omega_{\phi}^{-1} = \cO_{\bP^1}(-1-a)^2 \oplus 
\cO_{\bP^1}(-a)$$
and the splitting type is $(a,a+1,a+1)$.  
\item
Let $X \subset \bP^1\times \bP^4$ be a smooth complete 
intersection cut out by two forms 
$$F_1=L_{11}t_1+L_{12}t_2, \quad F_2=L_{21}t_1+L_{22}t_2, \quad
\bP^1=\bP^1_{[t_1,t_2]},$$
of bidegree $(1,1)$ and 
one form $G$ of bidegree $(a-1,2)$. 
Assume that 
$Y:=\{F_1=F_2=0\}$ is generic, i.e.~the $L_{ij}$ are linearly independent. It follows that
$$Y\simeq \bP(\cO_{\bP^1} \oplus \cO_{\bP^1}(-1)^2).$$
We find
\begin{equation} \label{koszul:Y}
0 \ra \cO_{\bP^1 \times \bP^4}(-2.-2) \ra
\cO_{\bP^1 \times \bP^4}(-1,-1)^2 \ra \cO_{\bP^1 \times \bP^4}
\ra \cO_Y \ra 0
\end{equation}
whence
$$0 \ra \cO_{\bP^1}(-1)^2 \ra \cO_{\bP^1}^5 \ra 
(\pi_1)_*(\cO_Y(0,1)) \ra 0$$
and 
$$(\pi_1)_*\cO_Y(0,1) \simeq \cO_{\bP^1} \oplus \cO_{\bP^1}(1)^2.$$
We find that
$$\phi_*\omega_{\phi}^{-1} = \cO_{\bP^1}(-1-a) \oplus \cO_{\bP^1}(-a)^2$$
with splitting type $(a,a,a+1)$.  
\end{enumerate}
\end{exam}
\begin{rema}
On first glance, the splitting type $(0,0,1)$ is missing.  
However for $Y=\{F_1=F_2=0\} \subset \bP^1 \times \bP^4$,
the restriction homomorphism
$$\Gamma(\cO_{\bP^1 \times \bP^4}(-1,2)) \ra 
\Gamma(\cO_Y(-1,2))$$
fails to be surjective; this follows by computing the 
twist of (\ref{koszul:Y}) by $\cO_{\bP^1\times \bP^4}(-1,2)$. 
Indeed
$$\Gamma(\cO_Y(-1,2)) \simeq \Gamma(\cO_{\bP^1}(-1) \oplus
\cO_{\bP^1}^2 + \cO_{\bP^1}(1)^3),$$
which is eight-dimensional.  These divisors all contain
the distinguished section of $Y \ra \bP^1$. 
\end{rema}
\begin{exam} \label{exam:unbalanceP4}
Consider the complete intersection
$$Y_0 = \{L_1t_1 - L_2t_2 = L_2t_1 - L_3t_2 = L_1L_3-L_2^2=0\}
\subset \bP^1 \times \bP^4,$$
with $L_1,L_2,L_3\in \Gamma(\cO_{\bP^4}(1))$ linearly independent. It is also smooth and is the image of 
$$\bP(\cO_{\bP^1}^2 \oplus \cO_{\bP^1}(-2)) 
\hookrightarrow \bP(\cO_{\bP^1}^5)\simeq \bP^1 \times \bP^4$$
induced by global sections.  
Conic bundles obtained from $Y_0$ are not balanced.
\end{exam}

\begin{exam} \label{exam:types2}
Over algebraically closed fields, representative balanced good conic bundles may be obtained via blowing up
$\beta: X \ra \bP^1 \times \bP^1$. Let $f_1$ and $f_2$
denote fibers of projections to $\bP^1$.
We follow the
taxonomy of Example~\ref{exam:types}.
In each case, the center $\{s_1,\ldots,s_n\}$
of $\beta$ consists of points in distinct
fibers of the first projection; 
$\phi:X\ra \bP^1$ is the composed morphism.
\begin{enumerate}
\item{Choose points $s_1,\ldots,s_{3a}$ imposing independent conditions
on the linear series $|2f_2+af_1|$. Then 
$\phi: X=\Bl_{s_1,\ldots,s_{3a}}(\bP^1 \times \bP^1) \ra \bP^1$
yields a good conic fibration of type $(a,a,a)$.}
\item{With points $s_1,\ldots,s_{3a+2}$
imposing independent conditions on $|2f_2+af_1|$,
$X=\Bl_{s_1,\ldots,s_{3a+2}}(\bP^1 \times \bP^1) \ra \bP^1$
is a fibration of type $(a,a+1,a+1)$.}
\item{With points $s_1,\ldots,s_{3a+1}$
imposing independent conditions on $|2f_2+af_1|$ we 
get a fibration 
$X=\Bl_{s_1,\ldots,s_{3a+1}}(\bP^1 \times \bP^1) \ra \bP^1$
of type $(a,a,a+1)$.}
\end{enumerate}
\end{exam}

\begin{exam}[Balance and birationality] \label{exam:warning}
Retain the notation of Proposition~\ref{prop:birat}.
Then $\phi:X \ra \bP^1$ may be balanced even when 
$\varphi:X'\ra \bP^1$ fails to be so.
For example, consider 
$$s_1,s_2,s_3,s_4 \in \bP^1 \times \bP^1, \quad
\pi_2(s_1)=\pi_2(s_2)$$
but otherwise generic. Take 
$$X:=\Bl_{s_1,s_2,s_3,s_4}(\bP^1\times \bP^1)\stackrel{\beta}{\longrightarrow}
\Bl_{s_1,s_2}(\bP^1\times \bP^1)=:X',$$
with structure morphisms $\phi$ and $\varphi$
projecting onto the first factor.
Proposition~\ref{prop:birat} gives
$$\cO_{\bP^1}(-1)^2 \oplus \cO_{\bP^1}(-2) \simeq \phi_*\omega_{\phi}^{-1} 
\hookrightarrow \varphi_*\omega_{\varphi}^{-1}
\simeq \cO_{\bP^1}^2 \oplus \cO_{\bP^1}(-2).$$
In particular, $\phi:X \ra \bP^1$ is balanced
despite having a section with self-intersection $(-2)$ i.e. the horizontal ruling
containing $s_1$ and $s_2$.  
\end{exam}

\begin{exam}[Balance and Singularity]
\label{exam:warning2}
Now consider $X$ obtained by blowing up four
points on the diagonal
$$s_1,s_2,s_3,s_4 \in \Delta_{\bP^1} \subset \bP^1 \times \bP^1 $$
but otherwise generic. Let $D\subset X$ denote the
proper transform of the diagonal, with $D^2=-2$. We obtain an embedding
$$X \hookrightarrow 
\bP( (\phi_*\omega_{\phi}^{-1})^{\vee})\simeq
\bP(\cO_{\bP^1}(1)^2 \oplus \cO_{\bP^1}(2))
\simeq \bP(\cO_{\bP^1}(-1)^2 \oplus \cO_{\bP^1})
\hookrightarrow \bP^1 \times \bP^4.$$
However, on projection to $\bP^4$, $D$ is contracted 
to a point $x_0$;
this morphism comes from the linear series 
of bidegree $(2,2)$ forms on $\bP^1\times \bP^1$
vanishing at $s_1,s_2,s_3,s_4$. 
The image of $X$ in $\bP^4$ has an ordinary double point
at $x_0$.
\end{exam}

\section{Lattices and Weyl groups}
\label{sect:LWG}

Let $\phi: X \ra \bP^1$ be a good conic bundle over an algebraically closed field with $n\ge 1$ degenerate fibers. Write $f$ for the class of a
fiber of $\phi$ and take $\Sigma$ to be a section of $\phi$, which
exists by Tsen's Theorem. Write $E'_1,E''_1,\ldots,E'_n,E''_n$ for the irreducible components of the degenerate fibers so that
$$f \equiv E'_1 + E''_1 \equiv \cdots \equiv E'_n+E''_n;$$
we label so that
$$\Sigma \cdot E'_i=1, \Sigma \cdot E''_i=0, \quad i=1,\ldots,n.$$
We have
$$(E'_i)^2=(E''_i)^2=-1, \quad E'_i\cdot E''_i=1,$$
with the other intersection numbers among $\{E'_1,\ldots, E''_n\}$ equal to zero.  Note that $\{f,\Sigma,E''_1,\ldots,E''_n\}$ freely generates
$\Pic(X)$ and 
$$K_X = -(r+2)f -2\Sigma + E''_1 + \cdots + E''_n,
\quad
\Sigma^2=-r.$$
\begin{rema} \label{rema:goodbasis}
Example~\ref{exam:types2} gives special cases with
$$f=f_1, \Sigma=f_2, E''_i=E_i, E'_i=f_1-E_i$$
so that $r=0$ and
$$K_X = -2 f_1 - 2 f_2 + E_1+\cdots + E_n.$$
The basis $\{f_1,f_2,E_1,\ldots,E_n\}$ may be 
used for all good conic fibrations with a section of even self-intersection, 
not just those arising as blowups of $\bP^1\times \bP^1$. Indeed, we may take
$$f=f_1, \Sigma=f_2-\frac{r}{2}f_1, E''_i=E_i, E'_i=f_1-E_i.$$
\end{rema}

\begin{defi} \label{defi:parity}
Let $\phi:X \ra \bP^1$ be a good conic bundle
with $n\ge 1$ degenerate fibers
$$E_1'\sqcup E_1'',\ldots,E_n' \sqcup E''_n.$$
The collection of disjoint fibral curves
$$\{E_1', \ldots,E'_n\}$$ 
is {\em even} (resp.~{\em odd}) if the 
rational ruled
surface obtained by blowing it down
$$\beta: X \ra \bF_r$$
has $r$ even (resp.~odd). 
\end{defi}
When $r$ is even (resp.~odd) then sections of $\bF_r \ra \bP^1$ all have even (resp.~odd) 
self-intersection. 
Note that if
$\{E_1',E_2'\ldots,E'_n\}$ is even then
$\{E_1'',E_2',\ldots,E'_n\}$ is odd, 
as the blow-down associated with
the second collection is isomorphic to
$\bF_{r-1}$ or $\bF_{r+1}$.

\

We record a classical result (see, for instance, \cite{SkorobogatovBIRS,HassettRSNF}) on intersections:
\begin{prop} \label{prop:getweights}
Consider $\Pic(X)$ as a unimodular lattice under the intersection form.
For $n\ge 2$,
$$
R_n:=\left< K_X,f\right>^{\perp} \subset \Pic(X)
$$
is isomorphic to the root lattice for $D_n$ and its extension
$$W_n:=\Pic(X)/\left<K_X,f\right>$$
is isomorphic to the weight lattice.
The group
$$\{ g\in \Aut(\Pic(X)): g(f)=f, g(K_X)=K_X\}$$
equals the Weyl group $W(D_n)$.
\end{prop}
\begin{proof} The intersection
pairing induces 
$$R_n \subset W_n \subset R_n \otimes \bQ.$$
We use the basis of Remark~\ref{rema:goodbasis}, so
that
$$R_n = \left<f_1-E_1-E_2,E_1-E_2,E_2-E_3,\ldots,
E_{n-1}-E_n \right>,
$$
which has the desired intersections.  
Take the images of
$$ E_1, \ldots, E_n, f_2$$
as generators of $W_n$; setting
$$ L_i:=E_i - \frac{1}{2}f_1 \in R_n \otimes \bQ$$
we have 
$$
\frac{L_1+\cdots + L_n}{2}  \equiv f_2 \pmod{\bQ K_X + \bQ f_1}.
$$
This corresponds to the weights of $D_n$, in the notation 
of \cite[\S 18.1,19.2]{FulHar}; the roots are $\pm L_i \pm L_j,
i\neq j$, i.e. differences of components in distinct singular fibers.
The Weyl group $W(D_n)$ in the basis $\{L_i\}$ consists of signed permutations with an even number of $-1$'s; this coincides with
permutations of irreducible components of degenerate fibers, compatible with 
intersections and preserving parities 
of collections (see Definition~\ref{defi:parity}) .

We turn to the last assertion of the Proposition. Observe that the
discriminant groups
$$R_n^{\vee}/R_n \simeq \left<K_X,f\right>^{\vee} / \left<K_X,f\right>
$$
are isomorphic to $(\bZ/2\bZ)^2$ (for $n$ even) or $\bZ/4\bZ$ (for $n$ odd). Elements of the Weyl group
act trivially on discriminant group. 
Every 
automorphism of $R_n$ acting as the identity on the discriminant
group $R_n^{\vee}/R_n$ is in the Weyl group \cite[V.11]{SerreLie}. 
\end{proof}
\begin{rema} 
For $n\ge 3$, the full group $\Aut(R_n)$ is a semidirect
product of $W(D_n)$ with the automorphisms of the Dynkin diagram 
$D_n$ \cite[V.11]{SerreLie}. The latter factor encodes outer automorphisms 
of the associated complex Lie group \cite[VI.3]{SerreLie}: 
These are the symmetric group $\fS_3$ for $D_4$ and the cyclic group $C_2$ for the other $n$.
\end{rema}

Leaving the conic fibration unspecified yields more
complicated indefinite root systems:

\begin{exam} \label{exam:Mukai}
The blowup $X' \ra \bP^1 \times \bP^1$ at $n$ sufficiently generic
points is a good conic bundle with $n$ degenerate fibers, with respect to 
either of the projections to $\bP^1$.
Now $K_{X'}^{\perp}$ has basis
$$f_2-f_1, f_1-E_1-E_2, E_1-E_2, E_2-E_3,\ldots, E_{n-2}-E_{n-1},
E_{n-1}-E_n,$$
where the $f_i$ are the classes associated with fibrations over $\bP^1$
and $E_1,\ldots,E_n$ are the exceptional classes. 
This lattice is associated with the root system $T_{3,n-2,2}$:

\begin{figure}
$$
\xymatrix{
{\bullet} \ar@{-}[r] & {\bullet} \ar@{-}[r] & {\bullet} \ar@{-}[r] \ar@{-}[d]
& {\bullet} \ar@{-}[r] & {\bullet} \ar@{-}[r] & {\bullet} \ar@{-}[r] & {\bullet} \ar@{-}[r] &{\bullet} \ar@{-}[r] & {\bullet}\\
                      &                     & {\bullet} 
&                       &                       &   & & & 
}
$$
\caption{$T_{3,7,2}$ diagram for $n=9$}
\label{fig:Tpqr}
\end{figure}
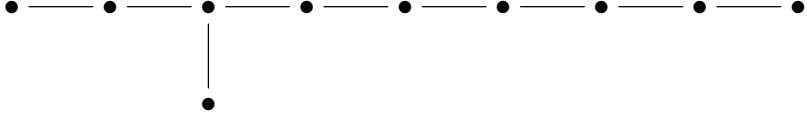
The diagram $T_{3,n-2,2}$ is a tree of $(-2)$-curves with chains
of lengths $3,n-2$, and $2$ attached to the unique trivalent 
vertex (Figure~\ref{fig:Tpqr}). See \cite[p.~127]{MukCremona} 
but note the typos in the formulas. The subgroup fixing
$f_1$ is isomorphic to $W(D_n)$.  
\end{exam}

Let $G=\Spin_{2n}$ denote the simply-connected semisimple linear
algebraic group associated with the Dynkin diagram $D_n$. 
Consider the half-spin representations \cite[\S 20.1]{FulHar}
$$\rho_{\pm}:\Spin_{2n} \ra \GL_{2^{n-1}}.$$
They both have weights taken from
\begin{equation}
\left\{ \frac{\pm L_1 \pm L_2 \cdots \pm L_n}{2} \right\}; \label{eqn:signs}
\end{equation}
the weight sets are those consisting of elements
with even and odd numbers of $-1$'s;
these sets are conjugate under an outer automorphism.  
\begin{prop} \label{prop:halfspin}
The half-spin representations have weights corresponding to 
$$\{E \in \Pic(X): E^2=-1, K_XE=-1, f\cdot E =1 \}$$
and
$$\{C \in \Pic(X): C^2=0, K_XC=-2, f\cdot C =1 \},$$
i.e. numerical exceptional curves and conics of degree one with respect to 
$\phi$.  
\end{prop}
\begin{proof}
The weights of each half-spin representation form a single orbit
under the $W(D_n)$-action. On the other hand, the geometric 
descriptions of the exceptional curves and conics are also compatible
with the $W(D_n)$-action. Thus it suffices to compute the image
in the weight lattice for a single class of each type. 

This is easy to check using the basis of Remark~\ref{rema:goodbasis}
and the identification of Proposition~\ref{prop:getweights}.
The first collection corresponds to
$$f_2-E_{i_i}, f_2+f_1 - E_{i_1} - E_{i_2} - E_{i_3}, \ldots,
f_2+kf_1 - E_{i_1} - \cdots - E_{i_{2k+1}}, \ldots,$$
where $1\le i_1 < \cdots < i_{2k+1} \le n$ and $2k+1 \le n$. In other words, these are expressions (\ref{eqn:signs}) whose number of positive signs
has parity different from $n$. 
The second collection is
$$f_2, f_2+f_1 - E_{i_1} - E_{i_2}, \ldots,
f_2+kf_1 - E_{i_1} - \cdots - E_{i_{2k}}, \cdots, 2k \le n,$$
which are expressions (\ref{eqn:signs}) whose number of positive signs
has the same parity as $n$.  
\end{proof}
\begin{rema}
The notation of Proposition~\ref{prop:halfspin} suggests
that the sets are represented by $(-1)$-curves and 
conic fibrations. Degenerate cases
of Example~\ref{exam:types2} make clear what can go wrong: The points $\{s_i\}$ might not have distinct
images under the second projection or may fail
to impose independent conditions on the linear series
$f_2+kf_1$. 
\end{rema}

\section{Moduli, Cremona transformations, and monodromy}
\label{sect:MCTM}
We continue to assume that $\phi: X \ra \bP^1$ is a good conic bundle over an algebraically closed field with $n\ge 1$ degenerate fibers.

\subsection{Infinitesimal automorphisms}
\begin{prop} \label{prop:balancenoVF}
If $\phi:X \ra \bP^1$ is a good balanced conic bundle,
with $n\ge 3$ degenerate fibers, then $\Gamma(X,T_X)=0$.
\end{prop}
\begin{proof}
Let $\Aut^{\circ}(X)$ denote the identity component of the 
autmorphism group; its tangent space, at the identity, is
$\Gamma(X,T_X)$. Since $\phi$ has at least three degenerate
fibers, automorphisms in $\Aut^{\circ}(X)$ must commute with
$\phi$ i.e. $\Aut^{\circ}(X) = \Aut^{\circ}(X/\bP^1)$. 
Thus vector fields in $\Gamma(X,T_X)$ lie in the 
fibers of $\phi$.

Since $\phi:X\ra \bP^1$ is a good conic bundle, 
it admits a birational morphism
$$\beta:X \ra \bF_r,$$
blowing down $(-1)$-curves in distinct
fibers over $\bP^1$ (see Proposition~\ref{prop:disc}).
Let $s_1,\ldots,s_n\in \bF_r$ denote 
the center of $\beta$. We may assume that none of these
points are contained in a negative section 
of $\bF_r \ra \bP^1$ (if there is such a section).
Indeed, if $\phi$ admits a section of negative
self-intersection, $\beta$ blows down curves disjoint from
that section. 
Finally, Proposition~\ref{prop:birat} and the balanced
condition implies that $r \le \lfloor (n+2)/3\rfloor$.  

For $n=4$, classification shows that the 
following conditions are equivalent:
\begin{itemize}
\item{$X$ is balanced;}
\item{$\phi$ has at most one section of self-intersection
$(-2)$;}
\item{$X\simeq \Bl_{s_1,s_2,s_3,s_4}(\bP^1 \times \bP^1)$
where the $\pi_1(s_i)$ are distinct and 
at most two points have the same image under $\pi_2$;}
\item{$\Gamma(X,T_X)=0$.}
\end{itemize}
Below, we assume $n\neq 4$.

Suppose that $r=0$. Then there are non-zero vector
fields on $\bF_0$ vanishing at $s_1,\ldots,s_n$ only
if these points are contained in two
horizontal rulings. Now at least half the points are
on one of the rulings. The proper transform $\Sigma$
of this ruling in $X$ satisfies $-\Sigma^2 \ge n/2$;
Proposition~\ref{prop:birat} implies 
$-\Sigma^2 \le \lfloor (n+2)/3 \rfloor$.
These are inconsistent unless $n=4$.  

For $r>0$, any vector fields necessarily vanish along the 
negative section, of self-intersection $-r$.  
Suppose $\bF_r$ has non-zero vector fields vanishing
at $s_1,\ldots,s_n$; this can happen only if they are all contained 
in a section with self-intersection $r$. 
Its proper transform $\Sigma \subset X$ would have self
intersection 
$$\Sigma^2 = r-n \le \lfloor (n+2)/3 \rfloor - n,$$
again contradicting Proposition~\ref{prop:birat}.
\end{proof}

\subsection{Parametrizing conic bundles}
Choose an even collection of $n$ disjoint irreducible
components of degenerate fibers, as 
in Definition~\ref{defi:parity}. 
Blowing down yields a birational morphism over $\bP^1$
\begin{equation} \label{eqn:chooseeven}
\beta: X \ra \bF_r, \quad r \text{ even},
\end{equation}
with center $s_1,\ldots,s_n \in \bF_r$, points in distinct
fibers over $\bP^1$.  

\begin{prop} \label{prop:atlas}
Good balanced conic bundles
$\phi: X\ra \bP^1$
with $n$ degenerate fibers may be parametrized
by a smooth irreducible scheme of finite type.
The generic such $X$ is isomorphic to a blowup 
of $\bF_0$.
\end{prop}
The balance assumption is necessary for finite type:
Otherwise, we would have blow-ups of $\bF_r$ for any even $r$!
\begin{proof}
The splitting type of $\phi_*\omega^{-1}_{\phi}$
is determined by $n$, the degree of the discriminant
(see Proposition~\ref{prop:disc}). By Proposition~\ref{prop:birat}, sections $\Sigma$ of $\phi$ have
self-intersections bounded from below. Specifically,
$\beta:X \ra \bF_r$
for $r \le \lceil n/3 \rceil + 1$. Thus our conic
bundles are parametrized by a scheme of finite type.
The deformation space of $\bF_r$ is smooth
of dimension $r/2$; the space of $n$-tuples of distinct points on these surfaces in also smooth.  
Thus the parameter space is smooth and irreducible. The 
description of the generic member follows because a 
generic deformation of $\bF_r$ (for even $r$) is isomorphic to $\bF_0$.
\end{proof}

\begin{defi}
A good conic bundle $\phi:X \ra \bP^1$
with $n\ge 1$ degenerate fibers
is {\em typical} if every even collection 
of curves blows down to $\bF_0$.

A typical conic bundle $\phi:X \ra \bP^1$ with
$n\ge 3$ degenerate fibers
is {\em uniform} if, for each $\beta:X \ra \bF_0$
arising from an even collection
and indices $1\le i_1<i_2<i_3 \le n$, the
points $\beta(E'_{i_1}), \beta(E'_{i_2}), \beta(E'_{i_3})\in \bF_0$
are in general position.

A uniform conic bundle is {\em uniformly balanced} 
if each blow down $\beta:X \ra X'$ along a collection
of $<n$ fibral curves is balanced.  
\end{defi}
These conditions are all stable under blow downs
of fibral curves
$$
\xymatrix{ X \ar[rr]^{\beta} \ar[rd]
& & X' \ar[ld] \\
& \bP^1 & 
}
$$
even though being balanced is {\em not} respected
under blow down (Example~\ref{exam:warning}).

\begin{prop}
Consider good conic bundles with $n\ge 1$ degenerate
fibers.  
\begin{itemize}
\item{Balanced, typical, uniform, and uniformly balanced are dense Zariski-open conditions on good conic bundles
with the requisite number of degenerate fibers.}
\item{A uniform conic bundle $X\ra \bP^1$ with $n$
degenerate fibers, equipped with
an ordered collection $\{E'_1,\ldots,E'_n\}$,
has no automorphisms. }
\end{itemize}
\end{prop}
\begin{proof}
The first assertion follows from Proposition~\ref{prop:atlas}.
Any automorphism of $\bP^1\times \bP^1$ fixing 
three points in general position is the identity, which
gives the second assertion. 
\end{proof}
There may be automorphisms acting non-trivially
on the fibral curves (see Proposition~\ref{prop:n4aut}).

\subsection{Introduction of moduli stacks}
\begin{defi}
The moduli stack $\cC_n$ (resp.~$\cC_n^b$, $\cC_n^t$, $\cC_n^u$ or $\cC_n^{ub}$) parametrizes
$$(\phi:X \ra \bP^1; E'_1,\ldots,E'_n)$$
where $\phi$ is a good (resp.~balanced, typical, uniform, or uniformly balanced) conic bundle
with $n$ degenerate fibers and
$E'_1,\ldots,E'_n$ is an even tuple of 
components of degenerate fibers. 
Isomorphisms are commutative diagrams
$$
\xymatrix{ X \ar^{\sim}[r] \ar_{\phi}[d] & Y \ar^{\varphi}[d] \\
\bP^1  \ar^{\sim}[r] & \bP^1
}
$$
where the horizontal arrows are isomorphisms
respecting the fibral curves and their order.
\end{defi}
Except for $\cC_n$, these stacks are of finite type.
The group $C_2^{n-1}$ acts on all these spaces via relabeling
the components of the degenerate fibers i.e.~for an even
number of indices we replace $E'_i$ with $E''_i$. 
The Weyl group $W(D_n)$ acts transitively on even collections
of fibral curves; see the proof of Proposition~\ref{prop:getweights} for an explanation why.
Being good, balanced, typical, uniform, and uniformly balanced are compatible with 
this action, thus $W(D_n)$ acts naturally on $\cC_n, \cC^b_n$, $\cC^t_n$,
$\cC^u_n$, and $\cC^{ub}_n$.

\begin{prop}
Consider the diagonal action of $\PGL_2\times \PGL_2$
on $n\ge 3$ copies of $\bF_0\simeq \bP^1 \times \bP^1$.  
The mapping
\begin{align*}
\iota_n: \cC^t_n &\hookrightarrow [\PGL_2\times \PGL_2 \backslash \bF_0^n] \\
(\phi:X \ra \bP^1; E'_1,\ldots,E'_n) & \mapsto
(\beta(E'_1),\ldots,\beta(E'_n))
\end{align*}
is an open embedding. 
\end{prop}

For generic $s_1,s_2,s_3 \in \bF_0^3$ 
there is a unique 
$\alpha \in \PGL_2\times \PGL_2$ with
$$s_1 \stackrel{\alpha}{\mapsto} (0,0), \
s_2 \stackrel{\alpha}{\mapsto} (1,1), \
s_3 \stackrel{\alpha}{\mapsto} (\infty,\infty).$$
Thus we have another open embedding
\begin{align*}
\bF_0^{n-3} & \hookrightarrow [\PGL_2\times \PGL_2 \backslash \bF_0^n]\\
(s_4,\ldots,s_n) & \mapsto [(0,0),(1,1),(\infty,\infty),
s_4,\ldots,s_n].
\end{align*}

\begin{prop}
The embedding $\iota_n$ restricted to the uniform locus
is also an open embedding
\begin{align*}
j_n: \cC^u_n &\hookrightarrow \bF_0^{n-3} \\
(\phi:X \ra \bP^1; E'_1,\ldots,E'_n) & \mapsto
(\alpha(\beta(E'_4)),\ldots,\alpha(\beta(E'_n))).
\end{align*}
\end{prop}

The spaces $\cC^u_n$, and $\cC^{ub}_n$ are open subsets of
$\bF_0^{n-3}$, which inherits a birational
action of the Weyl group. 

\begin{rema}
Constructions of birational actions of 
generalized Weyl groups on products of projective spaces are discussed in \cite[ch.~VI]{DolgOrt};
the formalism there does not directly
address our case. Mukai \cite{MukCremona}
extended this framework; the case relevant
for us is Example~\ref{exam:Mukai}.
Section 2 of \cite{CTHilb14} (see Lemmas 2.2 and 2.3) explains how $W(D_n)$
arises via small birational modifications
$$\Bl_{q_1,\ldots,q_n}(\bP^{n-3}) \stackrel{\sim}{\dashrightarrow} \Bl_{q_1,\ldots,q_n}(\bP^{n-3}),$$
where $q_1,\ldots,q_n \in \bP^{n-3}$ are in general
position. 
\end{rema}

Quotients of 
quasi-projective varieties under
finite groups (like $W(D_n)$) are separated stacks
with quasi-projective coarse moduli spaces. 
We refer the reader to \cite[App.~1]{MustataZeta} for discussion, including
behavior of quotients under field extensions.

We summarize our results:
\begin{theo} \label{theo:modulistack}
The moduli stack of uniform (resp.~uniformly balanced)
conic fibrations
$\phi:X \ra \bP^1$ with $n\ge 3$ degenerate fibers
is isomorphic to $[W(D_n) \backslash \cC^u_n]$ 
(resp.~$[W(D_n)\backslash \cC^{ub}_n]$). These are separated
with quasi-projective coarse moduli spaces.
\end{theo}
\begin{coro} \label{coro:monodromy}
The monodromy representation on Picard
groups of good conic bundles
with $n\ge 3$ degenerate fibers
is the Weyl group $W(D_n)$.  
\end{coro}
\begin{proof}
Theorem~\ref{theo:modulistack} reflects that
each even $n$-tuple $\{E'_1,\ldots,E'_n\}$
canonically determines a blowup representation
$\beta:X \ra \bF_0$.  
Proposition~\ref{prop:getweights} guarantees
the monodromy is at most $W(D_n)$. A dense Zariski-open subset of moduli has a geometrically-connected $W(D_n)$ cover, i.e. $\cC^u_n$; since the monodromy over $\cC^u_n$
is trivial by construction, so the corollary follows.
\end{proof}
\begin{rema}
Over infinite fields, unirational varieties
have Zariski-dense rational points.
Thus we can
always find points $s_1,\ldots,s_n \in \bF_0$ such
that the resulting conic bundle is uniformly balanced.
Further, the
classes of Proposition~\ref{prop:halfspin} may
be realized by irreducible $(-1)$-curves and
conic fibration. The configurations
where these fail are Zariski-closed proper
subsets of $\bF_0^n$. 
\end{rema}

The recursive structure of our definitions yields the following:
\begin{prop}
There exist natural dominant forgetting morphisms
\begin{align*}
\Phi: \cC_n & \longrightarrow M_{0,n} \\
(\phi:X \ra \bP^1,E'_1,\ldots,E'_n) & \mapsto 
(\bP^1,\phi(E'_1),\ldots,\phi(E'_n))
\end{align*}
and
\begin{align*}
\Psi: \cC^{ub}_n & \longrightarrow \cC^{ub}_{n-r}, \quad r=1,\ldots, n-3 \\
(\phi:X \ra \bP^1,E'_1,\ldots,E'_n) & \mapsto
(\varphi:X' \ra \bP^1,E'_1,\ldots,E'_{n-r})
\end{align*}
where $\beta:X\ra X'$ blows down $E'_{n-r+1},
\ldots,E'_n$.
There are $\binom{n}{r}2^r$ choices for $\Psi$
depending on which fibral curves are blown down.
\end{prop}

\subsection{Small discriminant cases}
The cases $n=1,2$ are not representative as the underlying surfaces $X$ have no moduli and positive-dimensional
automorphisms.  

When
$n=3$ the moduli space consists of a single
point, albeit with non-trivial twists:
Let $\phi: X \ra \bP^1$ be a uniform conic bundle with three degenerate fibers.  Then
$X\simeq \Bl_{s_1,s_2,s_3}(\bP^1 \times \bP^1)$
where the three points are in
distinct fibers of both projections.
The fibration is isomorphic to the morphism of moduli spaces of pointed stable curves
$$\varphi_5: \oM_{0,5} \ra \oM_{0,4}$$
forgetting the last point. Its automorphism
group is permutations of the first
four points
$$\Aut(X) = \fS_4 \simeq W(D_3).$$

For $n=4$, see the proof of Proposition~\ref{prop:balancenoVF}
for classification details. Here the moduli space is two-dimensional, with non-trivial generic stabilizer:

\begin{prop} \label{prop:n4aut}
A generic conic bundle $X \ra \bP^1$ with four degenerate
fibers over an algebraically closed field $k$ 
has automorphism group $C_2^3$.
\end{prop}
\begin{proof}
Recall that a quartic del Pezzo 
surface 
$$X = \{Q_0=Q_{\infty}=0 \} \subset \bP^4, \quad
Q_0,Q_{\infty} \in k[x_0,x_1,x_2,x_3,x_4]_2$$
admits a $C_2^4$ action; indeed, diagonalizing
$$Q_0=x_0^2+x_1^2+x_2^2+x_3^2+x_4^2, \quad
Q_{\infty}=a_0x_0^2+a_1x_1^2+a_2x_2^2+a_3x_3^2+a_4x_4^2$$ we may let $\pm 1$ act on each coordinate. The 
pencil $\left<Q_0,Q_{\infty}\right>$ contains five rank-four
quadrics $Q_1,\ldots,Q_5$; each gives a double cover 
$$g_i:X \ra \bP^1 \times \bP^1$$
branched over a $(2,2)$ curve $E_i$. Projecting to the factors gives conic fibrations
$$\phi_{\pm i}: X \ra \bP^1$$
with four degenerate fibers.

Which elements of $C_2^4$ are equivariant for a given
fibrations? The group acts on 
$$\{\phi_{\pm 1},\ldots, \phi_{\pm 5} \}$$ 
by switching an even number of signs; the 
stabilizer of any element is isomorphic to $C_2^3$. 
The element 
$$\phi_{+i} \mapsto \phi_{-i}, \quad i=2,3,4,5$$
corresponds to the covering involution of $g_1$.
The remaining elements, in pairs indexed by partitions
$$\{2,3,4,5\} = \{2,3\}\sqcup \{4,5\}, \{2,4\}\sqcup \{3,5\}, \{2,5\} \sqcup \{3,4\},$$
correspond to two-torsion elements of the Jacobian
$J(E_1)$ acting on $E_1$. This interpretation remains
valid over non-closed fields.
\end{proof}

\section{Relations with parabolic bundles}
\label{sect:RWPB}
\subsection{Background}
Parabolic bundles on curves were introduced in
\cite{MS}; here we focus on the special case 
of rank-two bundles on $\bP^1$ with weight $1/2$
\cite{BHK,Casagrande}.

Fix a collection of distinct points $p_1,\ldots,p_n \in \bP^1$.
We consider rank-two bundles with parabolic structure
$$\mathbf{E}=(E,L_1,\ldots,L_n),$$
consisting of a rank-two vector bundle $E \ra \bP^1$ along with
one-dimensional subspaces $L_i \subset E|_{p_i}$ for $i=1,\ldots,n$.
Its parabolic degree is
$$\pardeg(\mathbf{E}) = \deg(E) + n/2.$$ We say that $\mathbf{E}$ is even or odd based on
the parity of $\deg(E)$, i.e. of the 
Hirzebruch surface $\bF_r:=\bP(E)$.

Given a rank-one subbundle $F \subset E$, the parabolic
degree of the associated $\mathbf{F} \subset \mathbf{E}$
is 
$$\pardeg(\mathbf{F}) = \deg(F) + \frac{1}{2} |\{i: L_i = F|_{p_i} \subset E_{p_i} \}|.$$
We say $\mathbf{E}$ is {\em semistable} (resp.~{\em stable}) if
$$\pardeg(\mathbf{F}) \le \text{(resp.~$<$) } \pardeg(\mathbf{E})/2$$
for every rank-one subbundle. It is {\em unstable} if it is not semistable. These are projective notions: 
Tensoring $\mathbf{E}$ by a line bundle does not change
its stability.  
For odd $n$, stability and semistability coincide.

\subsection{Extracting conic bundles}
\begin{defi} \label{def:partoconic}
Given $\mathbf{E}\ra \bP^1$ a parabolic bundle as above, the
associated good conic bundle is defined
$$X:=\Bl_{\bP(L_1),\ldots,\bP(L_n)}\bP(E) \ra \bP^1.$$
The group $C_2^{n-1}$ acts transitively on the
even parabolic projective bundles associated with 
$\phi:X \ra \bP^1$, by elementary (or Hecke) transformations over
an even subset of $\{s_1,\ldots,s_n\}$. Changing the 
ordering on these $n$ points, we obtain an action of 
$W(D_n)$ on these even parabolic projective bundles.
\end{defi}

\begin{rema}
We made the convention in (\ref{eqn:chooseeven}) to realize
conic bundles as blowups of {\em even} Hirzebruch surfaces.
This is a good choice because $\bF_0$ has simpler automorphisms than $\bF_1$.
However, we could have used odd Hirzebruch surfaces instead.
\end{rema}

\begin{prop}Consider the moduli space of 
$$(\bP(E) \ra \bP^1, p_1,\ldots,p_n)$$
of projective bundles with parabolic structure on 
$\{p_1,\ldots,p_n\}$, distinct points on $\bP^1$.  
The quotient under our $W(D_n)$-action gives a moduli space of good conic bundles over $\bP^1$ with $n$ degenerate fibers.
\end{prop}

\begin{defi}
A good conic bundle $\phi:X \ra \bP^1$ is {\em semistable} 
(resp.~{\em stable}) if all the associated 
even parabolic bundles $\mathbf{E} \ra \bP^1$ are 
semistable (resp.~stable).
\end{defi}

The general theory of parabolic bundles pays dividends in
our situation. Compactifications of universal moduli spaces
of parabolic bundles, over the moduli space of pointed curves,
have been constructed via GIT 
\cite{Schlueter,Pine}. These show that the moduli space
of stable parabolic bundles over the moduli space of pointed 
curves of genus zero has quasi-projective coarse moduli space. \begin{prop} \label{prop:stabmoduli}
The moduli space of stable good conic bundles with
$n$ degenerate fibers is separated with quasi-projective coarse
moduli space.
\end{prop}
\begin{proof} Let $\psi:\cP_n \ra \cC_n$ be the morphism of moduli stacks, from even stable parabolic projective bundles to good conic bundles with labeled degenerate fibers, arising from Definition~\ref{def:partoconic}.
This is equivariant under the actions of $C_2^{n-1}$ via elementary transformations on $\cP_n$ and relabelings on $\cC_n$. 
Let $\cP_n^s \subset \cP_n$ denote the stable locus and
and $\cC_n^s$ the image
$$\psi: \cap_{\tau \in C_2^{n-1}} \tau(\cP^s_n) \ra \cC_n.$$ 
This is also quasi-projective, and its quotients under
finite groups (like $W(D_n)$) remain quasi-projective\cite[App.~1]{MustataZeta}. 
\end{proof}

\begin{prop} \label{prop:baltostab}
Fix distinct points $p_1,\ldots,p_n \in \bP^1$.
Let $\mathbf{E} \ra \bP^1$ be parabolic 
of rank two and even degree with respect to these points;
write $\phi:X \ra \bP^1$ for the associated conic bundle.
If $n\ge 3$ and $\mathbf{E}$ is unstable (resp.~$n\ge 5$
and $\mathbf{E}$ is strictly semistable) then $\phi:X\ra \bP^1$ is unbalanced. 
\end{prop}
The edge cases for $n=4$ are discussed in Examples~\ref{exam:warning} and \ref{exam:warning2}.
\begin{proof}
In the balanced case, Corollary~\ref{coro:balanceforn} 
gives
$$(\phi_*\omega_{\phi}^{-1})^{\vee}=
\cO_{\bP^1}(\lfloor n/3 \rfloor) \oplus
\cO_{\bP^1}(\lfloor (n+1)/3 \rfloor) \oplus
\cO_{\bP^1}(\lfloor (n+2)/3 \rfloor).$$
By Proposition~\ref{prop:birat}, sections $\Sigma \subset X$ of $\phi$ satisfy
$$\Sigma^2 \ge -\lfloor (n+2)/3 \rfloor.$$

After tensoring by a line bundle, we may assume $\deg(E)=0$.
Assume that $\mathbf{E}$ is unstable with destabilizing
subsheaf $\mathbf{F} \subset \mathbf{E}$. 
Set 
$$k=|\{i:L_i=F|p_i\}|$$
so that being unstable translates to
$$k > \frac{n}{2} - 2\deg(F).$$
Let $\Sigma \subset X$ denote the proper transform of
$\bP(F) \subset \bP(E)$. Now
$$(\bP(F))^2_{\bP(E)}= - 2\deg(F)$$
and
$$\Sigma^2 = -2\deg(F) - k < -\frac{n}{2}.$$
Combining our inequalities involving $\Sigma^2$ yields
$$\lfloor (n+2)/3 \rfloor > \frac{n}{2},$$
a contradiction for $n\ge 3$. 
In the strictly semistable case we obtain
$$\lfloor (n+2)/3 \rfloor \ge \frac{n}{2},$$
a contradiction for $n\ge 5$.
\end{proof}

\begin{exam}
Stable parabolic bundles can yield unbalanced
conic bundles. Let $n=7$ and take points
$$s_1,\ldots,s_7 \in \bP^1 \times \bP^1, 
\quad \pi_2(s_1)=\pi_2(s_2)=\pi_2(s_3),
\pi_2(s_5)=\pi_2(s_6)=\pi_2(s_7)$$
but otherwise generic. Consider the conic bundle
$$\phi:X=\Bl_{s_1,\ldots,s_7}(\bP^1\times \bP^1)
\stackrel{\pi_1}{\longrightarrow}\bP^1$$
arising from the parabolic bundle
$$\mathbf{E} = (\cO_{\bP^1} \oplus \cO_{\bP^1},
L_1,\ldots,L_7) \quad \bP(L_i)=s_i,$$
which is stable. However, we have
$$(\phi_*\omega_{\phi}^{-1})^{\vee}= 
\cO_{\bP^1}(1) \oplus \cO_{\bP^1}(3)^2$$
as $s_1,\ldots,s_7$ sit on a reducible curve of 
class $f_1+2f_2$, where $f_i$ is the fiber of $\pi_i$.
\end{exam}

Since being balanced is a Zariski-open condition, Propositions~\ref{prop:stabmoduli} and \ref{prop:baltostab} yield:
\begin{coro}
For $n\ge 4$, the moduli stack of balanced conic bundles $[W(D_n)\backslash \cC_n^b]$
is separated with quasi-projective coarse moduli space.
\end{coro}
The $n=4$ case follows from the proof of Proposition~\ref{prop:balancenoVF}.

\begin{rema}
Moduli spaces of conic bundles over general curves
have also been studied in \cite{GSconic}. 
They develop formulates of 
stability grounded in slope-stability of vector
bundles.  Our approach gives very special instances
of these spaces -- even restricting to genus zero. 
\end{rema}

\subsection{Linear subspaces on complete intersections of quadrics}
Here we follow \cite{Casagrande}. Let $n=2g+1$ and consider
smooth complete intersections of quadrics 
$$Z = \{Q_1=Q_2=0 \} \subset \bP^{2g}$$
with pencil 
$$\cQ=\{t_1Q_1+t_2Q_2=0\} \ra \bP^1_{[t_1,t_2]}.$$
Write $p_1,\ldots,p_{2g+1} \in \bP^1$ for the singular members of the pencil.

Recall (see, for example, \cite[\S 2]{HKT}) that the primitive cohomology 
$$H^{2g-2}(Z,\bZ)_{\mathrm{prim}}$$ 
has $W(D_{2g+1})$ symmetry. Furthermore, the finite variety
$F_{g-1}(Z)$,
parametrizing $(g-1)$-dimensional projective 
subspaces isotropic for the pencil, also has this symmetry. 

Let $F_{g-2}(Z)$ denote the variety of $(g-2)$-dimensional
subspaces isotropic for $\cQ$. We have
$$\dim(F_{g-2}(Z))=2g-2=n-3.$$
Casagrande shows it is isomorphic to the moduli space of
stable bundles of rank two and degree zero, with parabolic
structure at $p_1,\ldots,p_{2g+1}$ of weight $1/2$.

Given $\bP(\Lambda) \in F_{g-2}(Z)$, we construct
a conic bundle $X \ra \bP^1_{[t_1,t_2]}$ with 
degenerate fibers over $p_1,\ldots,p_{2g+1}$.
Let $q$ be a non-degenerate quadratic form in $n$ variables over a field $K$ and $\Lambda$ an isotropic $K$-subspace of dimension $m$. Then $q|\Lambda^{\perp}$ is degenerate along $\Lambda$, inducing a 
non-degenerate quadratic form $q'$ on $\Lambda^{\perp}/\Lambda$, which has dimension $n-2m$. We say that $q'$ is the {\em Witt reduction} of $q$ for $\Lambda$. 

We apply this for $K=k(\bP^1)$, $q$ the generic fiber
of $\cQ \ra \bP^1$, $\Lambda$ the constant isotropic subspace.
Since $n=2g+1$ and $m=g-1$, the
Witt reduction is a ternary quadratic form. We take
$X\ra \bP^1$ to be the associated conic bundle.  

\section{Moduli of complete intersections}
\label{sect:MCI}

We follow the trichotomy presented in Example~\ref{exam:types},
giving complete-intersection descriptions of balanced 
good conic bundles.  

\subsection{Type (a,a,a):}
Consider $X=\{G=0\} \subset \bP^1 \times \bP^2$
smooth of bidegree $(a,2)$, of splitting type $(a,a,a)$
with $a\ge 1$.
The GIT quotient
$$ 
(\SL_2 \times \SL_3) \backslash \! \backslash \bP(\Gamma(\bP^1 \times \bP^2, \cO_{\bP^1\times \bP^2}(a,2)))
$$
is a birational model for balanced conic bundles
with $3a$ degenerate fibers. 
\begin{prop}
Smooth surfaces of this type have finite reduced
automorphism groups and are GIT stable. The resulting 
moduli stack $[W(D_{3a}) \backslash \cC^b_{3a}]$ for $a\ge 1$ is Deligne-Mumford
with quasi-projective coarse moduli space.
\end{prop}
Its proof follows the strategy of \cite[ch.~4, {\S}2]{GIT}
and \cite{BenoistJLMS}.  
The discriminant divisor 
$$\Delta \subset \bP(\Gamma(\bP^1 \times \bP^2, \cO(a,2)))$$
parametrizes singular surfaces. It is irreducible,
in fact, birational to a projective bundle over $\bP^1 \times \bP^2$; indeed, for $p \in \bP^1 \times \bP^2$, the
surfaces singular at $p$ are a codimension-four
linear subspace in the parameter space.
Since $\Delta$ is invariant under the group action,
the smooth surfaces are at least GIT semistable. 
Smooth $X \subset \bP^1
\times \bP^2$ of bidegree $(a,2)$ ($a>0$) are balanced.
By Proposition~\ref{prop:balancenoVF}, it has no global vector fields.

\subsection{Type (a,a+1,a+1):}
Consider 
$$X=\{F=G=0\} \subset \bP^1 \times \bP^3$$
be a smooth complete intersection of forms of bidegree $(1,1)$
and $(a,2)$ for $a\ge 1$. Now $F$ may be put in standard form: There are independent linear forms 
$L_1$ and $L_2$ on $\bP^3$ so that
$$F=L_1t_1+L_2t_2$$
and
$$Y=\{F=0\}\simeq \bP(\cO_{\bP^1}(-1) \oplus \cO_{\bP^1}^2).$$
In particular, $X$ is necessarily balanced and 
$\Gamma(X,T_X)=0$ by Proposition~\ref{prop:balancenoVF}.

Following \cite{BenoistJEMS},
these may be compactified by 
$$\bP(\cF) \ra \bP^7,$$
where $\cF \ra \bP^7$ is locally-free of rank $6a+10$
over the open locus of standard $F$.
We interpret 
$$[F] \in \bP(\Gamma(\cO_{\bP^1\times \bP^3}(1,1)))=\bP^7$$
and 
$$[G] \in \bP\left(\Gamma(\cO_{\bP^1 \times \bP^3}(a,2))/[F]\cdot
\Gamma(\cO_{\bP^1 \times \bP^3}(a-1,1))\right).$$
However, unlike in the previous case there is no
canonical way to linearize a GIT problem whose
stable locus contains all smooth $X$. See 
\cite[\S 2]{BenoistJEMS} and \cite{Benoistcodim2}
for discussion of the birational geometry.

\subsection{Type (a,a,a+1):}
Consider 
$$X=\{F_1=F_2=G=0\} \subset \bP^1 \times \bP^4$$
a smooth complete intersection of two forms of bidegree
$(1,1)$ and one of bidegree $(a-1,2)$. This is more
complicated than the previous cases.

\begin{rema} \label{rema:noKoszul}
The formalism of complete intersections is poorly adapted to this case. Assuming the forms $F_1$ and $F_2$
are generic, we may choose coordinates so that
\begin{equation} \label{eqn:generic11}
F_1 = L_{11}t_1 + L_{12}t_2, \quad 
F_2 = L_{21}t_1 + L_{22}t_2
\end{equation}
so that the image of
$$Y = \{F_1=F_2=0\}$$
in $\bP^4$ satisfies
$$\pi_2(Y) = \{L_{11}L_{22} = L_{12}L_{21} \}.$$
This relation is not in the ideal $\left<F_1,F_2\right>$
but is in its saturation with respect to
$\left<t_1,t_2\right>$. We obtain an alternate
resolution to (\ref{koszul:Y}):
\begin{align*} 0 \ra \cO_{\bP^1 \times \bP^4}(-1,-2)^2 
\ra \cO_{\bP^1\times \bP^4}(-1,-1)^2 \oplus 
\cO_{\bP^1\times \bP^4}(-2,-2) & \\
\ra \cO_{\bP^1\times \bP^4} \ra \cO_Y \ra &0,
\end{align*}
which is not a Koszul complex. 
\end{rema}

\begin{rema}
We cannot hope for vanishing of vector fields, as in Proposition~\ref{prop:balancenoVF},
for all such $X$. Recall Example~\ref{exam:unbalanceP4},
where 
$$Y_0=\{L_1t_1-L_2t_2=L_2t_1-L_3t_2=0\}$$
has extra automorphisms. When $a=1$ and $X\subset Y_0$, we have
$\Gamma(T_X)\neq 0$ even when $X$ is smooth. Here the
image in $\bP^4$
$$\pi_2(X) = \{L_1L_3-L_2^2 = G(L_1,\ldots,L_5)=0\}$$
generally has two ordinary
singularities on the line $L_1=L_2=L_3=0$. The morphism
$X\ra \pi_2(X)$ resolves these singularities; explicitly
$$X=\Bl_{s_1,s_2,s_3,s_4}(\bP^1\times \bP^1),\quad
\pi_2(s_1)=\pi_2(s_2), \quad \pi_2(s_3)=\pi_2(s_4).$$
The group $\bG_m$ acts on $X$, i.e., automorphisms of the 
second $\bP^1$ fixing the two distinguished points.
\end{rema}

\begin{ques}
Do  $[W(D_{3a+1}) \backslash \cC^b_{3a+1}]$ and
$[W(D_{3a+2})\backslash \cC^b_{3a+2}]$, for $a\ge 1$, 
have explicit constructions via complete intersections?
\end{ques}

\bibliographystyle{alpha}
\bibliography{conicmoduli}

\end{document}